\documentclass{amsart}
\usepackage[dvips]{color,graphicx}
\usepackage{amssymb,amsmath,amsthm,amstext,amsfonts, fancybox ,amscd,amsbsy}
\usepackage{pst-all}
\usepackage[active]{srcltx}
\usepackage[all,cmtip]{xy}
\usepackage{enumitem}
\usepackage{subfigure}

\usepackage[T1]{fontenc}

\usepackage[colorlinks=true,linkcolor=red,urlcolor=black]{hyperref}

\makeatletter
\def\referencia#1#2{\begingroup
    #2%
    \def\@currentlabel{#2}%
    \phantomsection\label{#1}\endgroup
}
\makeatother

\newtheorem{maintheorem}{Theorem}

\newtheorem{theorem}{Theorem}[section]

\newtheorem{definition}{Definition}[section]
\newtheorem{lemma}{Lemma}[section]
\newtheorem{corollary}{Corollary}[section]
\theoremstyle{remark}
\newtheorem{remark}{Remark}

\theoremstyle{plain}

\newtheorem*{theoremC}{Theorem \AE}

\newcommand{\F}{\ensuremath{\mathcal{F}}}

\newcommand{\diam}{\mathrm{diam\, }}
\newcommand{\htop}{\mathrm{h_{top}}}
\newcommand{\dist}{\mathrm{dist\, }}
\newcommand{\Fc}{\ensuremath{\mathcal{F}^c}}
\newcommand{\Fs}{\ensuremath{\mathcal{F}^s}}
\newcommand{\Fu}{\ensuremath{\mathcal{F}^u}}

\newcommand{\Fcs}{\ensuremath{\mathcal{F}^{cs}}}
\newcommand{\Fcu}{\ensuremath{\mathcal{F}^{cu}}}

\newcommand{\uFc}{\ensuremath{\tilde{\mathcal{F}}^c}}
\newcommand{\uFs}{\ensuremath{\tilde{\mathcal{F}}^s}}
\newcommand{\uFu}{\ensuremath{\tilde{\mathcal{F}}^u}}

\newcommand{\uFcs}{\ensuremath{\tilde{\mathcal{F}}^{cs}}}
\newcommand{\uFcu}{\ensuremath{\tilde{\mathcal{F}}^{cu}}}
\newcommand{\uF}{\ensuremath{\tilde{\mathcal{F}}}}

\newcommand{\Tord}{\mathbb{T}^d}
\newcommand{\Reald}{\mathbb{R}^d}
\newcommand{\tW}{\tilde{W}}

\begin{document}

\title{Invariance of Entropy for maps isotopic to Anosov}

\author{Pablo D. Carrasco}
\address{ICEx-UFMG, Av. Ant\^{o}nio Carlos, 6627 CEP 31270-901, Belo Horizonte, Brazil}
 \email{pdcarrasco@gmail.com}

\author{Cristina Lizana}
\address{IME - Universidade Federal da Bahia. Av. Adhemar de Barros s/n, Ondina, CEP: 40170-110, Salvador-Ba. Brazil}
 \email{clizana@ufba.br}
% \thanks{\noindent $\dag$
%This work was partially supported by TWAS-CNPq and ULA}

\author{Enrique Pujals}
\address{CUNY, 365 Fifth Avenue New York, NY 10016 USA}
 \email{epujals@gc.cuny.edu}
%\thanks{\noindent $\dag$
%This work was partially supported by PROSUL }

\author{Carlos H. V\'{a}squez}
\address{Pontificia Universidad Cat\'olica de Valpara\'{\i}so, Blanco Viel 596,
Cerro Bar\'on, Valpara\'{\i}so-Chile.}
\email{carlos.vasquez@pucv.cl}
\thanks{C.H.V.  was partially supported by Proyecto Fondecyt 1171427.}


\date{\today}

\begin{abstract}
We prove the topological entropy remains constant inside the class of partially hyperbolic diffeomorphisms of $\mathbb{T}^d$ with simple central bundle (that is, when it decomposes into one dimensional sub-bundles with controlled geometry) and such that their induced action on $H_1(\Tord)$ is hyperbolic. In absence of the simplicity condition we construct a robustly transitive counter-example.
\end{abstract}

\maketitle

%%%%%%%%%%%%%%%%%%%%%%%%%%%%%%%%%%%%%%%%%%%%%%%%%%%%%%%%%%%%%%%%%%%%%%%%%%%%%%%%%%%%%%%%%%%%%%%%%%%%%%%%%%%%%%%%%%%%%%%%%%%%%%%%%%%%%%%%%%%%%%%%%
\section{Introduction}\label{sec.introduction}

When studying the homotopy class of a given map one is often interested in finding the simplest model available. What simplicity means in a given context is of course a subjective matter, but a reasonable candidate could be the map in such a class for which the complexity of its orbits is minimal with respect to some appropriate measure. Following Mike Shub, who originally considered these type of questions \cite{Shub1974} we will express this quantitatively by  the notion of \textit{topological entropy}.  Roughly speaking, fixed an error $\epsilon>0$ and a  time $n\geq 1$, an ``experimental orbit'' is a set of points  such that  until time $n$ they are separated  up to $\epsilon$ by the iterates of $f$. The topological entropy of $f$ is the rate of growth of the number of  different ``experimental orbits''  which can be observed  when the time $n$ goes to $+\infty$ and the error $\epsilon$ goes to $0$ (see Section~\ref{ssec:entropy} for a precise definition).

We restrict ourselves to differentiable diffeomorphisms of compact manifolds. Given an isotopy class $[f]$ we can ask the following.

\medskip

\noindent\textbf{Question:} Can we characterize the minimizers in $[f]$? Can we give sufficient conditions for a map $g\in[f]$ to be a minimizer of the topological entropy? 

\medskip

Even in this setting the question appears to be very hard, and a complete answer is not known for dimension greater than one. In the case of compact surfaces we have, due to the Nielsen-Thurston classification, that given an (orientation preserving) diffeomorphism $f,$ there exists $g$ homotopic to it satisfying one of the following:

\begin{itemize}
\item $g^p$ is the identity for some $p\in\mathbb{N}$, or
\item $g$ is pseudo-Anosov, or
\item $g$ leaves invariant some finite set of closed simple curves.
\end{itemize}  
Periodic maps have zero entropy, so those are minimizers of the entropy inside their isotopy class. For maps in the class of pseudo-Anosovs, we have the following result.

\begin{theorem}[Fathi-Shub, \cite{Fathi2012}]
Let $S$ be a compact surface and $f: S\rightarrow S$ be a diffeomorphism in the isotopy class of a pseudo-Anosov $A: S \rightarrow S$. Then $\htop(f) \geq \htop(A)$. 
\end{theorem}

The third case is reducible (one cuts the surface along the invariant curves), and can be, in principle, understood by reducing to the other cases. Thus, in dimension two, the first part of the previous question has a definitive answer; that is, given an isotopy class $[f]$ we know how to find the simplest model in terms of entropy. 

To study these type of questions and generalizations to higher dimensions we are led to consider isotopy classes containing maps whose dynamics is well understood, and we take the approach here of considering isotopy classes of Anosov diffeomorphisms. From now on $M$ is a closed (compact, boundaryless) Riemannian manifold. 

Recall that a diffeomorphism $f:M\rightarrow M$ is Anosov if it has a $Df$-invariant splitting $TM=E^s\oplus E^u$ where $E^s$ is uniformly contracted and $E^u$ uniformly expanded (see Section~\ref{ssec:ph} for a precise definition).

\begin{definition}
A $C^1$ diffeomorphism $f: M \rightarrow M$ is called \emph{Derived from Anosov} (DA) if it is isotopic to an Anosov diffeomorphism. In the case $M=\Tord,$ then $f$ is isotopic to its action in homology $A:H_1(\Tord)\rightarrow H_1(\Tord)$. We call $A$ the linear part of $f$. 
\end{definition}

A consequence of the results obtained by J. Franks (cf. Section \ref{semiconjugation}) is that if $f:\Tord\rightarrow\Tord$ is a DA having linear part $A$, then $\htop(f) \geq \htop(A)$, thus in this category we already have the existence of a simple (linear) model minimizing the entropy inside the isotopy class. It is meaningful then to ask:

\medskip

\noindent\textbf{Question:} if $f:\Tord\rightarrow\Tord$ is a DA, when is $\htop(f)=\htop(A)$?

\medskip

Let us give an illustrative example of what we are discussing.

\begin{theoremC}
Consider a DA diffeomorphism $f:\mathbb{T}^2\rightarrow\mathbb{T}^2$ and assume that there exists an isotopy $(f_t)_{t\in[0,1]}$ between $f$ and its linear part such that for every $t\in[0,1]$ the diffeomorphism $f_t$ is partially hyperbolic, that is, it preserves a dominated splitting\footnote{We say that $TM=E\oplus F$ is a dominated splitting if there exist $n$ such that for every $x\in M$ it holds $\|Df^n(x)\mid_E\|\cdot\|Df^{-n}(f^n(x))\mid _F\|\leq 1/2$.} with $\|Df_t|F_t\|>1$. Then $\htop(f) = \htop(A)$. 	
\end{theoremC}

To avoid introducing more notation at this stage we relegate the (simple) proof of the above result to the Appendix. We will be concerned with some generalizations of this theorem to higher dimensions. Concretely, we  study partially hyperbolic DAs; we do so since this class is currently focus of very active research, and since it enjoys a fair amount of meaningful examples. Let us recall the definition of partial hyperbolicity.

\begin{definition}
A $C^1$ diffeomorphism  $f: M \rightarrow  M$\ is \emph{partially hyperbolic} if it has an invariant splitting of the form $TM=E^s\oplus E^c\oplus E^u$ where $E^u$ (resp. $E^s$) is uniformly expanded (resp. contracted), and $E^s\oplus E^c,E^c\oplus E^u$ are dominated.
\end{definition}

The bundles $E^s,E^u,E^c$\ are called the \emph{stable, unstable} and \emph{center} bundle respectively. Further details are given in Section \ref{ssec:ph}. The reader can also consult \cite{CHHU2018, HPsurvey, RH2Ures} for surveys in partial hyperbolicity.

An important example of DA was introduced by Ma\~{n}\'{e} in \cite{contributions}; it is obtained by deforming a fixed point of a linear Anosov $A:\mathbb{T}^3\rightarrow \mathbb{T}^3$. The resulting system is partially hyperbolic with one dimensional center bundle. In \cite{BFSV} it is shown that the deformed map has the same entropy than $A$. See also \cite{ures}, and \cite{RRoldan} for a more precise description in this context.

In the present article we consider the case of a partially hyperbolic DA $f:\Tord\rightarrow \Tord$ whose center bundle $E^c$ has arbitrary dimension, and, as it is usually the case for the examples, integrates to an invariant foliation $\Fc$ (called the center foliation). This apparent mild generalization in the conditions conduces in fact to much more intricate possibilities; while in the one-dimensional case the induced action on the center foliation corresponds to families of maps on lines, in our case we have to deal with maps on higher dimensional manifolds, which are much more delicate.

To obtain some positive results we furthermore assume some control in the geometry of the stable and unstable foliations $\Fs,\Fu$ (that these exist is consequence of the well known Stable Manifold Theorem) and of the center bundle as well, in particular we will be working in the usual setting when $E^{cs}=E^s\oplus E^c, E^{cu}=E^u\oplus E^c$ integrate to $f$-invariant foliations $\Fcs,\Fcu$ (the so called \emph{dynamically coherent case}). We will need the following definition. 

\begin{definition}
Let $N$ be a (necessarily non-closed) manifold and let $\mathcal{F}_1,\mathcal{F}_2$ be foliations of $N$ such that $T\mathcal{F}_1\oplus T\mathcal{F}_2=TN$. We say that $\mathcal{F}_1,\mathcal{F}_2$ have Global Product Structure (GPS) if 
\[
	x,y\in N\Rightarrow \#\mathcal{F}_1(x)\cap \mathcal{F}_2(y)=1.
\]
\end{definition}

The control assumed in $E^c$ is the following.

\begin{definition}\label{def.simple}
Let $f$ be a partially hyperbolic diffeomorphism. We say that its center bundle $E^c$ is simple if
 \begin{enumerate}
 \item[a)] $E^c=E^1\oplus\cdots \oplus E^\ell$ with $\dim E^i=1$, for every $i=1,\ldots, \ell$.
 \item[b)] For every $S\subset \{1,\cdots,\ell\}$ the bundle $E^S:=\oplus_{i\in S}E^i$ integrates to an $f$-invariant foliation $\mathcal{F}^S$ (in particular, $E^c=E^{\{1,\cdots,\ell\}}$ is integrable). Furthermore, there is compatibility in the sense: $S\subset S'\Rightarrow \mathcal{F}^S$ sub-foliates $\mathcal{F}^{S'}$.
\end{enumerate}
 We say that $E^c$ is strongly simple if it is simple and furthermore
 \begin{enumerate}
  \item[c)] For every $i$, the lifts of $\mathcal{F}^i:=\mathcal{F}^{\{i\}}, \mathcal{F}^{\{1,\cdots,\hat{i},\cdots,\ell\}}$ to the universal covering of $M$ have GPS inside each leaf of the lift of $\Fc$.
 \end{enumerate}
\end{definition}

We remark that in the above definition we are not requiring domination in the center bundle. For reference, we say the decomposition in $E^c$ (simple) is dominated if for any $i$, $\bigoplus_{j=1}^i E^j$ is dominated by $\bigoplus_{j=i+1}^{\ell}E^j$: in this case we write $E^c=E^1\oplus_<\cdots \oplus_< E^\ell$.

\smallskip

We state our first result giving some sufficient conditions providing an answer to the question above. 

\begin{maintheorem}\label{teo:A}
Let $f:\Tord\rightarrow \Tord$ be a DA partially hyperbolic diffeomorphism. Assume further that:
\begin{enumerate}
\item the lifts of $\Fcs,\Fu$ to $\mathbb{R}^d$ have GPS, and likewise for $\Fs, \Fcu$;
\item $E^c$ is strongly simple.
\end{enumerate}
Then $\htop(f)=\htop(A)$. If furthermore $E^c$ is dominated then the same is true for $\mathcal{C}^1$ small perturbations $g$ of $f$, provided that $g$ has simple center bundle.
\end{maintheorem}

Theorem \ref{teo:A} above is an extension of the main product of \cite{BFSV}; it will be consequence from a more general result, Theorem \ref{boundedgeo}, that we will state and prove in Section \ref{sec:BoundedGeom}. The hypotheses $(1)$ of GPS (as well as dynamical coherence) holds for example in the connected component of $A$ inside the set of partially hyperbolic diffeomorphisms of $\Tord$; see \cite{Fisher2014}. We remark however that here we are not assuming \emph{absolute partial hyperbolicity} (see the comments in the above cited article), neither are we assuming domination for $E^c$, a common working hypothesis when dealing with center bundles that split into (invariant) one-dimensional ones. We also point out that the fact that $E^c$ decomposes into one-dimensional sub-bundles does not imply in general that any of them is hyperbolic (thus preventing reductions to the case when the center is one-dimensional treated in the previous literature): see Theorem D in \cite{Bonatti2003} for open sets of transitive diffeomorphisms having this property. GPS in the center is the strongest requirement here: it holds for example if $E^c_A$ is simple and these foliations have continuations in the isotopy between $A$ and $f$.

Let us bring to the attention of the reader a subtle part in the previous theorem; the requirement of GPS on the foliations is not in principle an open condition. Nonetheless, the conclusion of the Theorem is open: even though it is not assumed GPS between $\Fcs_g,\Fu_g$ (or $\Fs_g, \Fcu_g)$, small perturbations $g$ of $f$ are also partially hyperbolic DA's with the same linear part; in the case when we assume further dominated splitting in the center of $f$ the same holds for $g$, and we can then conclude that the topological entropy remains constant provided that the center of $g$ has also a decomposition into one-dimensional foliations. This last geometrical requirement is necessary, since in principle these one-dimensional foliations are not dynamically defined.
   
\smallskip

It is worthwhile to analyze the necessity of the hypotheses, in particular of the one referring to the simplicity of the central bundle. This hypothesis turns out to be crucial.

\begin{maintheorem}\label{teo:B}
There exist $g:\mathbb{T}^4\rightarrow\mathbb{T}^4$ a partially hyperbolic DA with irreducible linear part $A$, $\mathcal{U}$ a $C^1$ neighborhood of $g$ and a positive constant $c$ such that for every $g'\in \mathcal{U}$ it holds
\begin{enumerate}
\item the lifts of $\Fcs_{g'},\Fu_{g'}$ to $\mathbb{R}^d$ have GPS, and likewise for $\Fs_{g'}, \Fcu_{g'}$;
\item $\htop(g')\geq\htop(A)+c$; and,
\item $g'$ is transitive. 
\end{enumerate}
\end{maintheorem}  

In other words, the topological entropy jumps in a robust way for a class of maps which are indecomposable in dynamical terms. It is possible to give simpler examples without the transitivity condition, although from the dynamical point of view these look somewhat artificial. On the other hand here we do not consider the measure-theoretical implications for these type of systems, and invite the interested reader to consider this problem. For results in the $3$-dimensional case see \cite{measurecenter}. 

The example in Theorem \ref{teo:B} is constructed by deforming a hyperbolic linear map, thus as we mentioned before, condition $(1)$ of Theorem \ref{teo:A} are satisfied. The key-fact that is lacking is the decomposition into one-dimensional sub-bundles; although we are requiring more for the proof of Theorem \ref{teo:A} (in particular, GPS inside the center leaves), it is possible that one can establish Theorem \ref{teo:A} only assuming the splitting of the center into one-dimensional sub-bundles. 

\medskip

\noindent\textbf{Question:} Does Theorem \ref{teo:A} hold assuming simplicity of the center bundle, or just that $E^c$ decomposes into invariant one-dimensional sub-bundles?

\smallskip
	
Compare \cite{Fisher2010} where the authors establish the existence of entropy maximizing measures for partially hyperbolic diffeomorphisms such that their center bundle splits into invariant one-dimensional sub-bundles. Some of the techniques of that article may be useful to answer the question above.

For partially hyperbolic DA's with two-dimensional center dimension (say, in the connected component of $A$ inside the partially hyperbolic ones with whose center is sub-foliated by one-dimensional ones) the results of this article cover essentially all dynamical interesting possibilities, except for the parabolic type behavior. One can ask the following.

\noindent\textbf{Question:} Assume that $f$ is deformation of $A$ inside the set of partially hyperbolic systems with two-dimensional center foliation $E^c$, and assume further that there exists $E\subset E^c$ one-dimensional invariant sub-bundle. Is $\htop(f)=\htop(A)$?

\smallskip 

The remainder of the article is organized as follows. In the next section we discuss some necessary preliminaries in entropy, partially hyperbolic dynamics and foliations. A very well known theorem due to J. Franks states the existence of a map $h:\Tord\rightarrow\Tord$ that semi-conjugates $f$ with its linear part. In Section~\ref{sec:structure} we analyze the structure of the sets $h^{-1}(x)$, and in Section~\ref{sec:BoundedGeom} we use that knowledge to deduce that the entropy of the map $f$ restricted to each one of these pre-image sets is zero. This, by Bowen's entropy formula, is enough to finish the proof of Theorem~\ref{teo:A}. In the last section we present the construction of the example described in Theorem~\ref{teo:B}. The article ends with a short Appendix containing the proof of Theorem \AE.

%%%%%%%%%%%%%%%%%%%%%%%%%%%%%%%%%%%%%%%%%%%%%%%%%%%%%%%%%%%%%%%%%%%%%%%%%
\section{Preliminaries}
In this section we review some basic notions and results that will be used throughout the paper.

\subsection{Entropy}\label{ssec:entropy}

Let $(X,\dist)$ be a metric space and $f:X\rightarrow X$ a uniformly continuous map. For $x\in X,\,n\in \mathbb{N},\,\epsilon>0$ we denote 
\[
B(x,n,{\epsilon})=\{y\in X: \max_{0\leq i\leq n-1} \textrm{dist}(f^ix,f^iy)<\epsilon\}.	
\]
Fix a set $K\subseteq X$. We say that a set $F\subseteq X$ $(n,\epsilon)$-\emph{spans} $K$ if 
\[K\subseteq\bigcup_{x\in F} B(x,n,{\varepsilon}).\]
If $K$ is compact we denote by $N(n,\epsilon,K)$ the minimum of the cardinalities of $(n,\epsilon)$-spanning sets for $K$.

\begin{definition} The \emph{topological entropy of $f$ on the compact set $K$} is 
\begin{equation}\label{eq:topent}
  \htop(f,K)=\lim_{\epsilon\rightarrow 0}\limsup_{n\rightarrow\infty} \frac{1}{n}\log N(n,\epsilon,K)\geq 0.
  \end{equation}
 The \emph{topological entropy of $f$} is
\[
\htop(f):=\sup_{K\subset X\ \mathrm{compact}} \{\htop(f,K)\}.
\]
\end{definition}
That the limit \eqref{eq:topent} exists is proven for example in \cite{Wal}. Alternatively, we can consider the following definition. A subset $E\subseteq K$ is called $(n,\epsilon)$-\emph{separated} if for $x\neq y\in E$ there exists $0\leq i\leq n-1$ such that $\textrm{dist}(f^ix,f^iy)>\epsilon$. We denote by $s(n,\epsilon,K)$ the largest cardinality possible of any $(n,\epsilon)$-separated set in $K$; we have (cf. \cite{B1971})
\begin{equation}\label{eq:topent2}
\htop(f,K)=\lim_{\epsilon\rightarrow 0}\limsup_{n\rightarrow\infty} \frac{1}{n}\log s(n,\epsilon,K).
 \end{equation}
Equality between the numbers given by \eqref{eq:topent} and \eqref{eq:topent2} is easily verifiable (cf. \cite{Wal}).
 
 Let $X$ and $Y$ be compact metric spaces, and let $f :X \to X$, $g: Y \to Y$, $\pi: X\to Y$ be continuous maps such that $\pi$ is surjective and $\pi\circ f = g\circ \pi$ (in other words, the maps $f$ and $g$ are semi-conjugated by the map $\pi$). Then Bowen \cite[Theorem 17]{B1971} proved that 
\begin{equation}\label{eq:Bowen}
\htop(f)\leq \htop(g)+\sup_{y\in Y}\htop(f,\pi^{-1}(y)).
\end{equation}

\subsection{Partial hyperbolicity}\label{ssec:ph}
We say that  a diffeomorphism $f\colon M\rightarrow M$  is \textit{partially hyperbolic} if there exists a continuous $Df$-invariant splitting $TM=E^s\oplus E^c\oplus E^u$ (we omit the reference to $f$ when no risk of confusion arises) and a Riemannian metric such that for every $x\in M$, for every unit vector $v^{\sigma}\in E^{\sigma}(x), \sigma=s,c,u$ it holds
\[
	\|Df(x)\cdot v^{s}\|<\|Df(x)\cdot v^{c}\|<\|Df(x)\cdot v^{u}\|.
\]
In this case we denote $\lambda_s:=\max_x\sup\{\|Df(x)\cdot v\|:v\in E^s(x),\|v\|\leq 1\}, \lambda_u:=\min_x\min\{\|Df(x)\cdot v\|:v\in E^u(x),\|v\|\leq 1\}$. The bundles $E^s,E^u,E^c$\ are called the \emph{stable, unstable} and \emph{center} bundle respectively. If $\dim E^c=0$, we say that $f$ is an \textit{Anosov diffeomorphism}. We also point out that the set of partially hyperbolic diffeomorphisms is  $\mathcal{C}^1$ - open (cf. \cite{HPS1977}). For partially hyperbolic diffeomorphisms, it is well-known that the stable and unstable bundles integrate to $f$-invariant foliations $\mathcal{F}^s$ and $\mathcal{F}^u$ \cite{HPS1977}. The leaf of $\mathcal{F}^{\sigma}$  containing $x$ will be denoted $W^{\sigma}(x)$, for $\sigma=s,u$. Such foliations are $f$-invariant and are contractive and expanding respectively, meaning that if we denote by ${\rm dist}_u,{\rm dist}_s$ the intrinsic metric in the corresponding leaf, then it holds for all $n\geq 0$,
\begin{equation}\label{eq:distsest}
y\in W^s(x)\Rightarrow {\rm dist}_{s}(f^nx,f^ny)\leq \lambda_s^n{\rm dist}_s(x,y),
\end{equation}
and
\begin{equation}\label{eq:distuest}
y\in W^u(x)\Rightarrow  \lambda_u^n{\rm dist}_u(x,y)\leq {\rm dist}_{u}(f^{n}x,f^{n}y).
\end{equation}

So far we have not discussed the integrability of $E^c$; in general this bundle is not integrable, even in the Anosov case \cite{SmaleBull}.  Here however, we will assume a stronger condition, which nevertheless it is usually satisfied in the examples.

\begin{definition}
A partially hyperbolic diffeomorphism $f$ is said to be dynamically coherent if both $E^{cs}=E^{s}\oplus E^c, E^{cu}=E^{c}\oplus E^u$ are integrable to $f$-invariant foliations $\mathcal{F}^{cs},\mathcal{F}^{cu}$. In this case intersecting the leaves of $\mathcal{F}^{cs},\mathcal{F}^{cu}$ we obtain an $f$-invariant foliation $\Fc$ tangent to $E^c$.
\end{definition}

The foliations $\mathcal{F}^{cs},\mathcal{F}^{cu}$ are called the center-stable and the center-unstable foliations of $f$. They are sub-foliated not only by leaves of $\Fc$ but also by leaves of $\Fs, \Fu$ respectively.

Before moving on let us mention that in the literature sometimes it is required quasi-isometry of the (lifts of) the strong stable and unstable foliations; assuming moreover that the system is absolutely partially hyperbolic (a more restrictive version of partial hyperbolicity) the hypotheses $(1)$ of Theorem \ref{teo:A} are satisfied. We are not assuming these conditions; we refer the reader to the surveys cited in the introduction and references therein for a more detailed discussion of this topic.

\smallskip

\noindent\textbf{Convention:} From now we fix $f$ satisfying the hypotheses of Theorem \ref{teo:A}; in particular
$M=\Tord$ and so its universal covering is $\tilde{M}=\Reald$. We add a tilde ($\tilde{\cdot}$) to denote the lift to $\tilde{M}$ (for example, $\uFc$ denotes the lift of the center foliation $\Fc$), with the exception of $f$ and $A$ since in this case no confusion would arise.  

We assume that the pairs of foliations $\uFs,\uFcu$ and $\uFu, \uFcs$ have GPS: for $x,y\in \Reald$ we denote
\begin{align*}
&\langle x,y\rangle_{csu}=\tW^{cs}(x)\cap\tW^u(y)\\
&\langle x,y\rangle_{cus}=\tW^{cu}(x)\cap\tW^s(y). 
\end{align*}

Let us recall the reader basic fact of foliation theory that if $\F=\{W(x)\}_{x\in M}$ is a foliation then, for $R>0$ fixed the family of discs $\{W(x,R)=\{y\in W(x):\dist_{\F}(x,y)\leq R\}\}_{x\in M}$ depends continuously on the base-point $x$. It is often said that $W(x,R)$ is a \emph{plaque} of the foliation $\F$.

\smallskip

 Denote $\pi:\tilde{M}\to M$ the projection and let $\tau:=\{T:\Reald\to\Reald: Tx=x+k, k\in \mathbb{Z}^d\}$; then $\tau$ is the Deck group of $\pi$, and in particular if $T\in \tau$,
 \begin{itemize}
 	\item $T$ is an Euclidean isometry;
 	\item $T$ preserves $\uF^{\sigma}, \sigma\in\{ s,c,u,cs,cu,i\}$.
 \end{itemize}
 As the metric on $M$ is induced by the one of $\tilde{M}$ via $\pi$, and $\pi\circ T=\pi$ we also have that:
 \begin{itemize}	
 	\item $T$ is an isometry for the intrinsic distance of $\tW^{\sigma}$.    
 \end{itemize}
 
 \smallskip

Recall the definition \ref{def.simple} of simple center bundle. For future reference we will record the following lemmas.

\begin{lemma}\label{lem.interseccionesuniformes}
Given $K>0$ there exists $L(K)<\infty$ such that 
\[
x,y\in \tilde{M}, \dist(x,y)<K\Rightarrow \dist(x,\langle x,y\rangle_{csu}), \dist(y,\langle x,y\rangle_{csu})<L(K).
\]
Similarly for the foliations $\uFcu,\uFs$.
\end{lemma}

 \begin{proof}
 We will prove more, and in fact establish that the intrinsic distances 
 \[
 	\dist_{cs}(x,\langle x,y\rangle_{csu}),\ \dist_{u}(y,\langle x,y\rangle_{csu})
 \] are uniformly bounded\footnote{Later we will only need the simplified statement of this Lemma.}.

 Assume by means of contradiction that there exists $K>0$ and $(p_n)_n, (q_n)_n\in \tilde{M}$ such that 
 \[
 \dist(p_n,q_n)\leq K, \text{ such that }z_n=\langle p_n,q_n\rangle_{csu}, \dist_{cs}(p_n,z_n)\geq n.
 \]

 Consider $(T_n)_n\in \tau$ such that $T_n^{-1}p_n\in [0,1]^d$; it is no loss of generality to assume that $T_n^{-1}p_n\xrightarrow[n\to \infty]{}p$, and by the local product structure between $\uFu$ and $\uFcs$ we can further assume $T_n^{-1}p_n\in \tilde{W}^u(p)$. Thus $\lim_n \dist(T_n(p),p_n)=0$ and for $n$ large, $\dist(T_n(p),q_n)\leq K+1$. Let $w_n=\langle T_n(p),q_n\rangle_{csu}$ (see figure \ref{fig:figure2}).

\begin{figure}[h]
	\centering
	\includegraphics[width=0.5\linewidth]{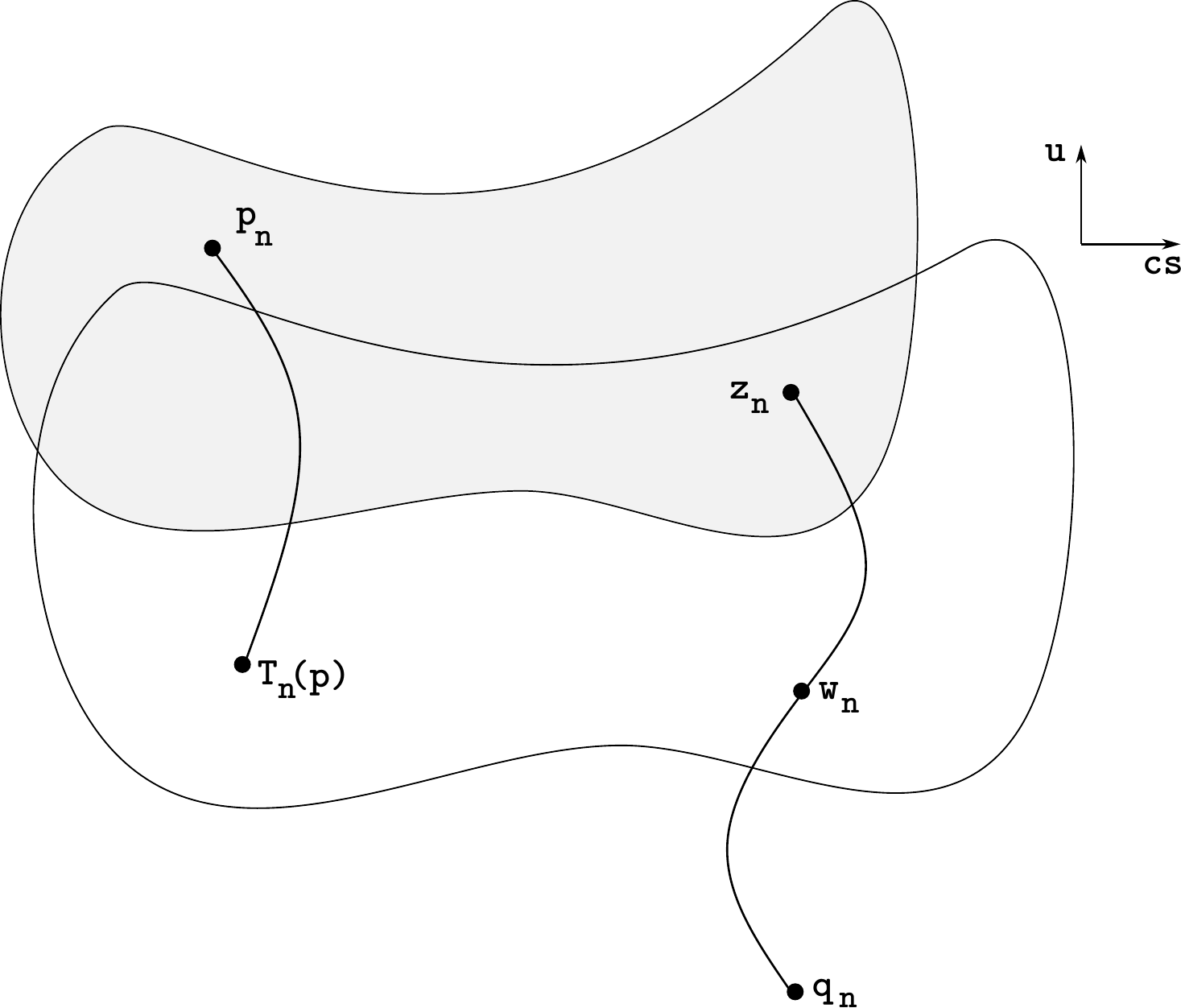}
	\caption{Diagram for the proof.}
	\label{fig:figure2}
\end{figure}

 \noindent\textbf{Claim:} $\limsup_n \dist_{cs}(T_n(p),w_n)=\infty$.

 If not, for some $R>0$ we have $w_n\in \tilde{W}^{cs}(T_n(p),R)=T_n(\tilde{W}^{cs}(p,R))$ for every $n\geq 0$. Since $\tW^{cs}(T_n^{-1}p_n,R)\xrightarrow[n\to\infty]{Haus} \tW^{cs}(p,R)$ we get that for $n$ sufficiently large the manifold $\tW^u(q_n)=\tW^u(w_n)$ intersects $\tW^{cs}(p_n,K)$. This gives a contradiction, since we had assumed that $\tW^u(q_n)$ intersects $\tW^{cs}(p_n)$ at $\dist_{cs}$ distance larger than $n$, for every $n$. 

 By the above argument we can suppose that $p_n=T_n(p)$. Now we argue analogously and find $(S_n)_n\in \tau$ such that 
 \begin{itemize}
  	\item $S_n(q_n)^{-1}\in [0,1]^d$;
  	\item $\lim_n S_n^{-1}(q_n)=q$;
  	\item $S_n^{-1}(q_n)\in \tilde{W}^u(q_n)$.
 \end{itemize} 
Note that for $n$ large, $\dist(T_n(p),S_n(q))\leq K+1$; it follows that if $R_n=T_n\circ S_n^{-1}\in \tau$, then  
$\dist(p,R_n(q))\leq K+1$ for $n$ large. Therefore, there exist finitely many $R_n$, this contradicts the fact 
$\lim_n \dist_{cs}(p,T_n^{-1}w_n)=\infty$.

Likewise for the $\dist_u$ distance; changing $f$ by $f^{-1}$ we get the result for $\uFcu,\uFs$. 
 \end{proof}

Let $\mathcal{K}(M)$ be the space of compact subsets of $M$, equipped with its natural Hausdorff distance. It is well known (cf. \cite[chapter 5]{HPS1977}) and simple to check that for  $\sigma\in \{s,c,u,cs,cu\}, K>0$ the maps $g\in\mathcal{C}^1\to \phi_g^{\sigma}\in \mathcal{C}^0(M,\mathcal{K}(M))$ are continuous, where
\begin{equation}\label{eq.depcontinuag}
	\phi_g^{\sigma}(x)=W^{\sigma}_g(x,K).
\end{equation}
In other words, for fixed radii $K$ the plaques $\{W^{\sigma}_g(x,K)\}_{x\in M}$ depend continuously on $g$. We deduce:

\begin{corollary}
Given $K>0$, there exists $\mathcal{U}$ a $\mathcal{C}^1$ neighborhood of $f$ and $L(K)>0$ such that if $g\in \mathcal{U}$ then $g$ is partially hyperbolic and
 \[
 x,y\in \tilde{M}\Rightarrow  w=\tilde{W}^{cs}_g(x)\cap \tilde{W}^{u}_g(y)\text{ with } \dist(x,w),\dist(y,w)<L(K).
 \]
Similarly for the foliations $\uFcu_g,\uFs_g$.
\end{corollary}

\begin{remark}
We are not saying that the lifted foliations of $g$ to $\tilde{M}$ have GPS.
\end{remark}

For $p,q\in \tilde{M}$ in the same leaf of $\uF^i$ let us denote by $[p,q]_i$ the interval inside $\tW^i(z)$ joining these two points.

\begin{lemma}\label{lem.weakqi}
Given $K>0$, there exists $\tilde{L}=\tilde{L}(K)$ such that 
\[
	p,q\in \tilde{M}, q\in \tW^i(q), \dist(p,q)\leq K\Rightarrow [p,q]_i\subset \tW^c(p,\tilde{L}).  
\]
\end{lemma}

\begin{proof}
The proof follows the same lines of the previous Lemma. Assuming this not being the case, there exists some $K>0$ 
and sequences $(p_n)_n, (q_n)_n$ such that $\dist(p_n,q_n)\leq K$, $[p_n,q_n]_i \not\subset \tW^c(p_n,n)$. By using translations (which are isometries for the intrinsic length of $\uF^i$) we can assume the existence of a sequence $(T_n)_n\in\tau$  such that 
\begin{itemize}
	\item $T_n^{-1}(p_n)\xrightarrow[n\to\infty]{}p$, for some $p\in [0,1]^d$;
    \item $\dist(T_n^{-1}(p_n),T_n^{-1}(q_n))\leq K$;
    \item $[T_n^{-1}(p_n),T_n^{-1}(q_n)]_i\not\subset\tW^c(q_n,n)$.
\end{itemize}
Since $T_n^{-1}(q_n)\in\tW^i(T_n^{-1}(p_n),K)$ and $\tW^i(T_n^{-1}(p_n),K)\xrightarrow[n\to\infty]{Haus}\tW^i(p,K)$ we can further assume
\begin{itemize}
	\item $T_n^{-1}(q_n)\xrightarrow[n\to\infty]{}q$, for some $q\in \tW^c(p,K)$.
\end{itemize}
Let $w=\tW^i(p)\cap \tW^{\{1,\cdots,\hat{i},\cdots,\ell\}}(q)$; by continuous dependence of $\uF^i$ on compact sets there exists $w_n\in \tW^i(T_n^{-1}(p_n))$ such that $[T_n^{-1}(p_n),w_n]_i\xrightarrow[n\to\infty]{Haus}[p,w]_i$, and in particular the intrinsic length of $[T_n^{-1}(p_n),w_n]_i$ is uniformly bounded, hence $\lim_n\dist_i(w_n,T_n^{-1}(q_n))=\infty$. By taking another subsequence we can assume that $w_n\xrightarrow[n\to\infty]{}w$, and then we can replace $T_n^{-1}(p_n)$ by $w_n$ and $p$ by $w$; note that $w\neq q$. But then for large $n$ necessarily $\tW^i(w_n)$ intersects $\tW^{\{1,\cdots,\hat{i},\cdots,\ell\}}(T^{-1}_n(q_n))$ in $w_n$, contradicting the GPS between $\uF^i$ and $\uF^{\{1,\cdots,\hat{i},\cdots,\ell\}}$ inside $\tW^c(w_n)$.
\end{proof}

\begin{remark}
If the leaves of $\Fc$ are simply connected the proof simplifies by using the compactness of the underlying manifold; however, this assumption leaves natural examples out of consideration, and we opted for giving the more general proof.
\end{remark}

\smallskip

\subsection{Semi-conjugation with the linear part}\label{semiconjugation}
 Now we discuss the fact that $f$ has hyperbolic linear part. Due to Franks \cite{F70} there exists a semi-conjugacy $h:M\rightarrow M$
%\Tord$ 
between $f$ and $A$, i.e.
\[\xymatrix{
M \ar[r]^{f} \ar[d]_{h} & M \ar[d]^{h}\\
M \ar[r]_A & M
%\Tord\ar[r]_A & \Tord
}\]
Its lift $\tilde{h}$ to $\tilde{M}$ semi-conjugates $\tilde{f}$ with $A$, and for some constant $K>0$, we have
\[
\|\tilde{h}-Id\|_{\mathcal{C}^0}\leq \frac{K}{2}. 
\]
In particular, $\tilde{h}$ is proper. We also remind the reader the following basic fact: if $g$ is a sufficiently small $\mathcal{C}^1$ perturbation of $f$, then $g$ is partially hyperbolic and furthermore is homotopic to $f$, hence in particular it has the same linear part $A$. The constant $K$ is uniformly bounded in a open $\mathcal{C}^1$ neighboorhood of $f$.

\begin{remark}\label{rmk.uniformetama}
For every $\tilde{x}\in \tilde{M}$, each $[\tilde{x}]$ is a compact set whose diameter is uniformly bounded from above
$ {\rm diam}([\tilde{x}])\leq K$. In particular, since $h\circ f^n=A^n\circ h$ for every $n\in\mathbb{Z}$ we also deduce 
\begin{equation}\label{claseacotada2}
{\rm diam}(\tilde{f}^n[\tilde{x}])\leq K, 
\end{equation}
for every  $n\in\mathbb{Z}.$ By the previous observation this is also true in a $\mathcal{C}^1$ neighborhood of $f$.
\end{remark}

In the next section, we analyze the structure of these sets $[x]$.

%%%%%%%%%%%%%%%%%%%%%%%%%%%%%%%%%%%%%%%%%%%%%%%%%%%%%%%%%%%%%%%%%%%%%%%%%
\section{Structure of the pre-image classes}\label{sec:structure}

We start noticing the following.

\begin{lemma}\label{dentrohoja} 
For every $x\in M$, the class $[x]$ is contained in a unique center leaf of $\Fc$. Likewise for a sufficiently small $\mathcal{C}^1$ perturbation of $f$.
\end{lemma}

\begin{proof}
It suffices to show that for every $\tilde{x}\in\mathbb{R}^d$ the class $[\tilde{x}]$ is contained in a unique leaf $\tilde{\mathcal{F}}^c$. Take $z,z'\in [\tilde{x}]$ and let $w=\langle z,z'\rangle_{csu}$. By Lemma ~\ref{lem.interseccionesuniformes}, for some $L<\infty$ it holds
\[
	\forall n\geq 0\ (n\in\mathbb{Z}),\ f^nw\in B(f^nz',L),
\]
but on the other hand $\dist_u(f^nz',f^nw)\geq \lambda_u^n\dist_u(z,w):=r_n$. By uniform continuity of $\uFu$ we deduce that for $n$ large the disc $\tW^u(f^nz',r_n)$ has to accumulate inside $B(f^nz',L)$, and therefore we will be able to find two disjoint plaques $\tW^u(t,r), \tW^u(t',r)\subset \tW^u(f^nz',r_n)$ that are close enough that $\tW^{cs}(t)\cap \tW^u(t',r)\neq\emptyset$. This contradicts the GPS between $\uFu$ and $\uFcs$. 

Therefore $w=z'$ and $[\tilde{x}]\subset \tW^{cs}(z)$. Arguing with $f^{-1}$ we deduce that $[\tilde{x}]\subset \tW^{cu}(z)\cap \tW^{cu}(z)=\tW^{c}(z)$.

Observe that all arguments are $\mathcal{C}^1$ robust: by Remark \ref{rmk.uniformetama} there  exists a $\mathcal{C}^1$ neighborhood $\mathcal{U}$ of $f$ and $K>0$ such that if $g\in\mathcal{U}$ then $\sup_{x\in M}\mathtt{diam}([\tilde{x}]_g)<K$. Given that the lifted foliations of $f$ have GPS, by \eqref{eq.depcontinuag} we get that the lifted foliations of $g$ have GPS for points $z,z'\in\tilde{M}$ with $\dist(z,z')<K$. Therefore by the same argument we conclude that $[\tilde{x}]_g$ is contained in a center leaf of $\uFc_g$, for $g\in\mathcal{U}$. 
\end{proof}

We can give a more precise characterization of $[x]$.  Recall that for $z\in \tilde{M}, z'\in \tilde{W}^i(z)$ we are denoting by $[z,z']_i$ the closed interval inside $\tilde{W}^i(z)$ with endpoints $z$ and $z'$. 

\begin{lemma}\label{intervalos}
For every $\tilde{x}\in \tilde{M}$, if  $z,z'\in [\tilde{x}]$ and  $z' \in \tW^i(z)$ for some $1\leq i\leq \ell$, then
\[ [z,z']_i\subseteq [\tilde{x}].\]
\end{lemma}

\begin{proof}
The proof is analogous to the previous Lemma. Consider $[z,z']_i$, by Lemma ~\ref{lem.weakqi} there exists some $L>0$ such that
\[
	\forall n\geq \mathbb{Z},\ [f^nz,f^nz']_i \subset \tW^c(f^nz,L).
\]
Since $A$ is hyperbolic, either $\tilde{h}([z,z'])=\{\tilde{x}\}$ or it is a closed non-trivial curve, and thus it growths under iterates of $A$, either for the past or the future. Therefore the intrinsic length of $[f^nz,f^nz']_i$ growths, which implies that it will accumulate inside $\tW^c(f^nz,L)$ contradicting the fact that $\uF^i$ intersects $\uF^{\{1,\cdots,\hat{i},\cdots,\ell\}}$ in at most one point.
\end{proof}

\begin{remark}
This argument would extend to $\mathcal{C}^1$ perturbations $g$ of $f$ provided that we knew that $\F^i$ exists and
\[
	g\to \{W^i_g(x,K)\}_{x\in M}
\]
is continuous in the Hausdorff topology; the continuity part is true for example if $E^c_f=E^1_f\oplus_<\cdots \oplus_< E^\ell_f$ is dominated (\cite[chapter 5]{HPS1977}). This is the only place where the domination condition on $E^c_f$ is used in the article. 
\end{remark}

\begin{corollary}\label{cor.intervalosstrong}
Assume that $f$ has a strongly simple center. Then there exists $\mathcal{U}$ a $\mathcal{C}^1$ open neighborhood of $f$ such that if $g$ in $\mathcal{U}$ has simple center, then for $1\leq i\leq \ell$ it holds
\[
\forall x\in M,	z,z'\in [x]_g, z' \in W^i_g(z)\Rightarrow [z,z']_i\subseteq [x]_g.
\]
\end{corollary}

\medskip

Fix  a point $c_0\in M$ and an integer $k$, $0\leq k\leq \ell$. We call a \emph{rectangle} (with corner $c_0$ and dimension $k$) to a compact set $R_k\subseteq W^c(c_0)$  obtained by the following inductive procedure. Let $c_0,\dots ,c_k$, where $c_j\in W^{i_j}(c_0)$; start with $R_1=[c_0,c_1]_{i_1}\subset W^{i_1}(c_0)$, then consider $[c_0,c_2]_{i_2}\subset W^{i_2}(c_0)$, $i_2\neq i_1$, and define $R_2$ as the trace inside $W^c(c_0)$ of the set obtained by sliding $R_1$ along $[c_0,c_2]_{i_2}$, that is,
\[
R_2=\bigcup_{x\in [c_0,c_2]_{i_2}} [x,y(x)]_{i_1},
\]
where $[x,y(x)]_{i_1}$ corresponds to the image of $[c_0,c_1]_{i_1}$ in $W^{i_1}(x)$ by the $\mathcal{F}^{i_2}$-holonomy. Having defined $R_{k-1}$ and given $[c_0,c_k]_{i_k}\subset W^{i_k}(c_0),$ $i_k\neq i_{k-1},$ define $R_k$ by sliding $R_{k-1}$ along $[c_0,c_k]_{i_k}$,
\[
R_k=\bigcup_{x\in [c_0,c_k]_{i_k}} R_x^{k-1},
\]
where $ R_x^{k-1}$ is (a rectangle of dimension $k-1$ and corner $x$) obtained as the image of $R_{k-1}$ in the corresponding center manifold  by the $\mathcal{F}^{i_k}$-holonomy sending $c_0$ in $x$.  The corners $c_0,\dots ,c_k$ as above define unequivocally the rectangle. 

Note that the order in which the rectangle is constructed is irrelevant. For example,
\[
R_2=\bigcup_{x\in [c_0,c_2]_{i_2}} [x,y(x)]_{i_1}= \bigcup_{z\in [c_0,c_1]_{i_1}} [z,w(z)]_{i_2}
\]
and similarly for higher dimensional rectangles.

We are ready to deduce the following.

\begin{corollary}\label{estructuraclase}\hfill
\begin{enumerate}
	\item If $f$ has simple center, then for every $x\in M$ the class $[x]$ is a rectangle in a single leaf of $\mathcal{F}^c$. Moreover, there exists some $E>0$ independent of $x$ such that for every interval $[z,z']_i\subset [x]$, it holds $\dist_i(z,z')<E$.
	\item If $f$ has strongly simple center, then the same is true for sufficiently small $\mathcal{C}^1$ perturbations $g$ of $f$.
\end{enumerate}	 
\end{corollary}

\begin{proof}
Let $z, z'\in[x]$. If $z\in W^{i_1}(z')$, then by Lemma \ref{intervalos}, we have that $R_1=[z,z']_{i_1}$ is contained in $[x]$. So, let us suppose that $z\not\in W^{i_1}(z')$, and denote by $W^{i_2}(z)$ the one dimensional center foliation such that $W^{i_1}(z')\cap  W^{i_2}(z)\neq \emptyset$. Let us denote by $w$ the point of intersection. We claim that  $w\in [x],$ furthermore, $R_2=\bigcup_{x\in [w, z]_{i_2}} [x,y(x)]_{i_1}$ is contained in $[x].$ 

\noindent\textbf{Claim:} there exists $E>0$ such that for $z,z' \in M, z'\in B(x,L)$ and $w=W^{i_1}(z')\cap  W^{i_2}(z)$  then 
\[
	\dist_{i_1}(z',w), \dist_{i_2}(z,w)\leq E.
\]
The proof is completely analogous to Lemma \ref{lem.interseccionesuniformes}, using compactness of $M$ and GPS inside $\F^{\{i_1,i_2\}}$.

\smallskip 
 
It follows as before that $w\in [x]$. Similarly, we get that $w'=W^{i_1}(z)\cap  W^{i_2}(z')\in [x]$. Using the previous Lemma we deduce that the rectangle $R_{2}$ of vertices $z,z',w,w'$ is completely contained in $[x]$. Note that the length of every interval $J\subset R_2$ inside $\F^{i_1}$ or  $\F^{i_2}$ has intrinsic length less than $E$.

Proceeding by induction, we have that if $w_{k-1}\in [x]\setminus R_{k-1},$ then there is $i_k$
such that the rectangle
\[R_k=\bigcup_{x\in [w, w_{k-1}]_{i_k}} R_x^{k-1},\]
is contained in $[x],$ where  $R_x^{k-1}$ is the image of $R_{k-1}$ in the corresponding center manifold by the $\mathcal{W}^{i_k}$-holonomy sending $w$ in $x$. The proof of the first part of the Lemma is complete.

The second part follows using continuous dependence of the foliations $\F^S_g$ on $g$ in the case when $f$ has strongly simple center, by arguing as in Lemma \ref{dentrohoja}.
\end{proof}

\smallskip

The previous Corollary shows that the structure of the classes $[x]$ is more delicate than the one-dimensional center case. We will use this structure in the next section to compute the entropy of $f$ inside a given class $[x]$ and show that it is equal to zero (see Theorem ~\ref{boundedgeo}). As a result, we deduce from \eqref{eq:Bowen} that
\[
\htop(A)\leq h_{\textrm{top}}(f)\leq \htop(A)+\sup_{x\in M}\{\htop(f,[x])\}=\htop(A).
\]
This establishes Theorem \ref{teo:A}.

%%%%%%%%%%%%%%%%%%%%%%%%%%%%%%%%%%%%%%%%%%%%%%%%%%%%%%%%%%%%%%%%%%%%%%%%%%%%%%%%%%%%%%%%%%%%%%%%%%%%%%%%%%%%%%%%%%%%%%%%%%%%%%%%%%%%%%%%%%%%%%%%%

\section{Bounded Geometry}\label{sec:BoundedGeom}

In this section we will  prove that $\htop(f,[x])=0$, for every $x\in M$. 

We will use the structure of the pre-image classes studied in Section~\ref{sec:structure}  and  we are going to  see such classes in a more general context. Let $X$ be a (not necessarily closed) Riemannian manifold, and consider an homeomorphism $f:X\rightarrow X$.

\begin{definition}
	We say that a $k$-dimensional compact subset $G\subset X$ is a geometrical wiring if there exist $1$-dimensional foliations $W^1,\ldots,W^k$ by compact leaves and $x\in G$ such that defining
		\[
		D_1:=W^1(x),\quad D_{i+1}:=\bigcup_{y\in D_i}W^{i+1}(y),\ 1\leq i\leq k-1,
		\]
		it holds
\begin{enumerate}
\item[(i)] $D_k=S$; and,
\item[(ii)] for every $i$, if $y\in D_i$, then $W^{i+1}(y)\cap D_i=\{y\}$.		
\end{enumerate} 	

\end{definition}

Follows from $(ii)$ above, that if $G\subset X$ is a geometrical wiring, then there exists a well defined projection $\pi_i^G:D_{i+1}\rightarrow D_i$.	

Assuming that $f^n(G)\subseteq X$ is a geometrical wiring, for every integer $n\geq 0$, we say that the corresponding foliations $W^1_{f^n(G)},\ldots,W^k_{f^n(G)}$ are $f$-invariant if 
$$
f(W^i_{f^n(G)}(x))=W^i_{f^{n+1}(G)}(fx), \quad\mbox{ for every } x\in G, i=1,\dots, k,\, n\geq 0.
$$ 
It follows that  $\{W^1_{f^n(G)},\ldots,W^k_{f^n(G)}:n\geq 0\}$ induce well defined $f$-invariant foliations  $W^1,\ldots, W^k$ on 
$$G_{\infty}:=\bigcup_{n\geq 0} f^n(G).$$

\begin{definition}
We say that a $k$-dimensional subset $G\subset X$ has \emph{bounded geometry} with respect to $f$ if the following conditions hold:
\begin{enumerate}[leftmargin=1cm]
\item[\referencia{BGuno}{\rm[BG1]}]  for every $n$, the set $f^n(G)$ is a geometrical wiring, and the corresponding foliations $W^1_{f^n(G)},\ldots,W^k_{f^n(G)}$ are $f$-invariant; 
\item[\referencia{BGdos}{\rm[BG2]}] the foliations  $W^1,\ldots, W^k$ are uniformly continuous;
\item[\referencia{BGtres}{\rm[BG3]}] the family $\{f:W^i(y)\rightarrow W^i(fy):y\in G_{\infty}\}$ is uniformly Lipschitz;
\item[\referencia{BGcuatro}{\rm[BG4]}] for every $x\in G_{\infty},$ there exists a parametrization $h^i_x:[0,1]\rightarrow W^i(x)$ such that the family  
\[
\mathcal{P}:=\{h^i_x:[0,1]\rightarrow W^i(x):1\leq i\leq k,x\in G_{\infty}\}
\]
is uniformly bi-Lipschitz.
\end{enumerate}	
\end{definition} 

\begin{remark}\label{rmk2}
Note that for a set as above, there exists $L>0$ such that for every $1\leq i\leq k,x\in S_{\infty}$, it holds $\mathrm{length}(W^i(x))<L$.	
\end{remark}

For $f$ as in Theorem \ref{teo:A} and for every $x\in M$, the class $[x]=h^{-1}(x)$ has bounded geometry with respect to $f$, as follows follows directly from Corollary \ref{estructuraclase}. 

\medskip

The main result of this section is the following, which implies in particular that for $x\in M$ holds $\htop(f,[x])=0$.

\begin{theorem}\label{boundedgeo}
If $G$ has bounded geometry with respect to $f$, then $\htop(f,G)=0$.	
\end{theorem}

\begin{proof} We fix then a set $S$ satisfying the hypotheses, and proceed  with a proof by induction on $k=\dim G$. 

By Remark \ref{rmk2}, for any interval $I$ inside $G$, for every $n\in \mathbb{N}$ it holds 
\[\mathrm{length}\big(f^n(I)\big)\leq L,
\] 
where $L$ is a uniform constant not depending on $n$ and $I$.  A classical argument then implies that 
\begin{equation}\label{unifN}
N(n,\epsilon,I)\leq n\left(\frac{L}{\epsilon}+1\right),	
\end{equation}
which in turn implies $\htop(f,I)=0$. This takes care of the base case $k=1$. 

For the inductive step,  we consider the $k$-dimensional rectangle with bounded geometry
\[
G=\bigcup_{z\in S_0}W^{k}(z), 
\] 
where $G_0$ be a $(k-1)$-dimensional rectangle with bounded geometry  such that $\htop(f,G_0)=0$.
For every integer $n\geq 0$ we denote $\pi_n: f^n(G)\rightarrow f^n(G_0)$ the corresponding projection. Thus, we can write for $n\geq 0$,
\[
f^n(G)=\bigcup_{z\in f^n(G_0)}W^k(z)=\bigcup_{z\in f^n(G_0)} J_z, 
\]
where $J_z$ is either a circle or a closed interval. Fix $\epsilon, r>0$. By \eqref{unifN} there exists $m=m(\epsilon,r)\in\mathbb{N}$ such that for every $z\in \cup_{n\geq0}f^n(G_0)$,  
\[
\frac{1}{m}\log N(m,\epsilon,J_z)<r.
\]
On each interval $J_z$ we consider $E_z$ an $(m,\epsilon)$-spanning set of $f|J_z$ with minimal cardinality and define the relative open set in $S$,
\[
U_z=\bigcup_{w\in E_z} B(w,m,3\epsilon).
\]

\smallskip

\noindent\textbf{Claim 1:} For $\epsilon>0$ sufficiently small, there exists $\eta>0$ such that for every $z\in \cup_{n\geq0} f^n(G_{0}),$ $w\in J_z, i\in\{1,\ldots,k-1\}$, the length of the connected component of  $B(w,m,\epsilon)\cap W^i(w)$ containing $w$ is greater than $ \eta$.

\begin{proof}[Proof of Claim 1]
Note that $B(w,m,\epsilon)=\bigcap_{j=0}^m f^{-j}B(f^j(w),\epsilon)$ and use \ref{BGdos}, \ref{BGtres}.  	
\end{proof}

For $z\in f^n(G_0)$, let $D_z\subset f^n(S_0)$ be the maximal (relatively) open disc centered at $z$ with the property $\pi^{-1}_n(D_z)\subset U_z$. 

\smallskip

\noindent\textbf{Claim 2:} We have that
\[
\inf\{\diam D_z:z\in \bigcup_{n\geq0}f^n(G_0)\}>0.
\]
\begin{proof}[Proof of Claim 2]
This follows from Claim 1, \ref{BGdos} and due to the fact that the cardinality $\#E_z$ is uniformly bounded independently of $z$ (recall that $E_z$ is $(m,\epsilon)$-spanning).	
\end{proof}

Using Claim 2 and \ref{BGcuatro} it is not difficult to deduce the existence of finite coverings $\mathcal{V}_n$ of $f^n(G_0)$ with the properties
\[
\delta(\epsilon):=\inf_{n\geq 0}\{\text{Lebesgue number of }\mathcal{V}_n\}>0, \quad \delta(\epsilon)\xrightarrow[\epsilon\mapsto0]{}0.
\]

Fix $p\in \mathbb{N}$ and take $F_p\subset G_0$ a $(p,\delta)$-spanning set  with minimal number of elements. Consider $q\in\mathbb{N}$ such that $q\cdot m< p\leq (q+1)\cdot m$. If $z\in F_p,$ then for every $0\leq j\leq p-1$ there exists $c_j(z)\in f^j(G_0)$ such that $B^{k-1}(c_j(z);\frac{\delta}{2})\subset D_{c_j(z)}$ and $D_{c_j(z)}\in\mathcal{V}_j$.

A list $(z,w_1,\ldots, w_{q})\in F_p\times J_{c_m(z)}\times J_{c_{2m}(z)}\cdots\times J_{c_{qm}(z)}$ will be called \emph{admissible}. For such a list define 
\begin{align*}
V(z,w_1,\ldots, w_q):=\{y\in G: \dist(f^{j+m}y,f^jw_s)<3\epsilon, \, 0\leq j\leq m, \,1\leq s\leq q\}	
\end{align*}
and observe that $\{V(z,w_1,\ldots, w_{q}): (z,w_1,\ldots, w_{q})\text{ is admissible}\}$ is an open covering of $G_{\infty}$. Note the following.

\smallskip

\noindent\textbf{Claim 3:} If $A\subset G$ is $(p,6\epsilon)$-separated and $(z,w_1,\ldots, w_{q})$ is admissible, then the cardinality  

$$\#\left(A\cap V(z,w_1,\ldots, w_{q})\right)\leq 1.$$

It follows then that, denoting by $M_z$ the number of possible admissible sequences of the form $(z,w_1,\ldots, w_{q})$, it holds 
\begin{equation}
s(p,6\epsilon,G)\leq \#F_p\cdot \max_{z\in F_p} M_z.	
\end{equation} 
On the other hand we have by the choice of $m$,
\begin{align*}
M_z=\prod_{s=0}^{q} N(m,\epsilon,J_{c_{sm}(z)})\leq \exp(qrm)\leq \exp(pr).
\end{align*}
Using the induction hypothesis we deduce that
\begin{align*}
\limsup_{p\rightarrow\infty} \frac{1}{p}\log s(p,6\epsilon,G)&\leq\limsup_{p\rightarrow\infty}\frac{1}{p}\log \#F_p+\limsup_{p\rightarrow\infty}\frac{1}{p}\log\exp(pr)\\
&\leq \htop(f,G_0)+r=r,
\end{align*}
and thus, $\htop(f,G)\leq r$. Since $r$ is arbitrary, we conclude that $\htop(f,G)=0$, which finishes the inductive step, and hence, the proof of Theorem \ref{boundedgeo}.
\end{proof}

\begin{remark}
The proof above is inspired on \cite[Theorem 17]{B1971}. The difficulty in our setting is that the set $G$ is not invariant, which forces us to contemplate its whole (positive) orbit.
\end{remark}

%%%%%%%%%%%%%%%%%%%%%%%%%%%%%%%%%%%%%%%%%%%%%%%%%%%%%%%%%%%%%%%%%%%%%%%%%

\section{Robustly transitive example}

In this section we exhibit a robustly transitive counter-example as we claimed in Theorem~\ref{teo:B}, showing the importance of the simplicity condition for establishing Theorem A. We point out that a construction of the same type appeared in \cite{RRoldan} without the robustly transitivity property, and with reducible central part.

Consider a totally real Pisot number $\alpha$ and let $A$ be the corresponding companion matrix: then $A:\mathbb{T}^4\to\mathbb{T}^4$ is a linear Anosov diffeomorphism with  decomposition
$$T\mathbb{T}^4=E_A^{sss}\oplus E_A^{ss}\oplus E_A^{s}\oplus E_A^{u},$$
with associated real eigenvalues $|\lambda^{sss}_A|<|\lambda^{ss}_A|<|\lambda^{s}_A|<1<|\lambda^{u}_A|$.
For instance, we can  consider the linear map with matrix
$$A=\left(
\begin{array}{cccr}
0 &0  &0  & -1 \\
1 & 0 &0 &-1\\
0 &1 &0 &10 \\
0 &0 &1 &10
\end{array}
\right),$$
where $\alpha=\lambda^u_A\thickapprox 10.91$, $\lambda^s_A\thickapprox -0.91$ and $\lambda^{1}_A\thickapprox -0.32, \lambda^{2}_A\thickapprox 0.32$. The topological entropy of $A$ is $\htop(A)=\log|\lambda^u_A|$.

Let us consider $p\in \mathbb{T}^4$ a fixed point by $A$ and $B_p$ a small neighborhood of $p$. We proceed to do a deformation in $B_p$ in order that the real eigenvalues associated to the central bundle $E_A^{ss}\oplus E_A^{s}$ become contracting complex eigenvalues, obtaining an isotopic diffeomorphism such that the central bundle is 2-dimensional and does not admit any invariant line sub-bundle. Let us call by  $f_0:\mathbb{T}^4\to\mathbb{T}^4$ the new map with
invariant splitting 
$$T\mathbb{T}^4=E^s_{f_0}\oplus E^{ws}_{f_0}\oplus E^{u}_{f_0},$$
where $E^s_{f_0}$ and $E^u_{f_0}$ are one dimensional and $E^{ws}_{f_0}=E_A^{ss}\oplus E_A^{s}$ (see below for this type of modification). The topological entropy of  $f_0$ is $h_{{\rm top}}(f_0)\geq h_{{\rm top}}(A)$. From now on, let us denote $E^{ws}_{f_0}$ by $E^c_{f_0}$.

Now, around $p$ we deform isotopically  $f_0$ ``\textit{a la Tahzibi-Bronzi}'' (see \cite{BT2010}) into a diffeomorphism $f$ having the following properties

\begin{enumerate}
\item[\referencia{Euno}{[E1]}] $f$ is partially hyperbolic with splitting 
\[
T\mathbb{T}^4=E^s_f\oplus E^{c}_f\oplus E^{u}_f,
\]
where $\dim E^s_f=\dim E^u_f=1$, and $E^c_f=E^c_{f_0}$. Hence, in particular, $Df|E^c_f$ does not admit a non-trivial invariant sub-bundle.
\item[\referencia{Edos}{[E2]}] The topological entropy $\htop(f)\geq \log|\lambda^u_A|+\log 2>\htop(A)$.
\item[\referencia{Etres}{[E3]}] $f$ is robustly transitive.
\end{enumerate}

We proceed as follows. For $r>0$ denote by $\mathbb{D}_r$ the $r$ disc in $\mathbb{R}^2$. As in \cite{BT2010} we construct an isotopy $\{h_t: \mathbb{D}_1\to\mathbb{D}_1\}_{t\in[0,1]}$ satisfying the following
\begin{enumerate}
    \item $h_t(z)=\lambda^{ws}_{f_0} z$, for all $t\in [0,1/2], |z|\geq 1/2$;

    \item $h_t(0)=0$, for all $t\in[0,1]$;
 
    \item $h_0$ is the Smale's horseshoe map (hence, $\htop(h_0)=\log 2$);
 
 	\item $h_t(z)=\lambda^{ws}_{f_0} z$, for all $t\in [1/2,1], z\in \mathbb{D}_1$.

\end{enumerate}

See Figure~\ref{fig:isotopy} for a view of $h_t$ on the disc $\mathbb{D}_1$ and Figure~\ref{fig:isotopyzoom} for a zoom inside the disc $\mathbb{D}_{1/2}$.

\begin{figure}[htbp]
\centering
\subfigure{\includegraphics[width=40mm]{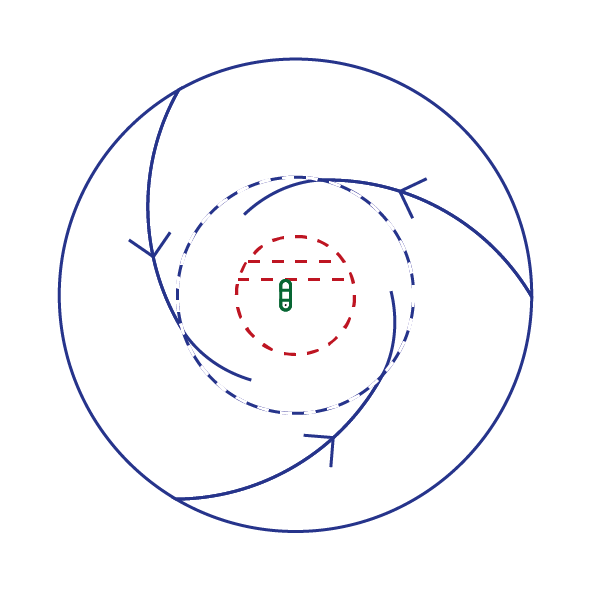}}
\subfigure{\includegraphics[width=40mm]{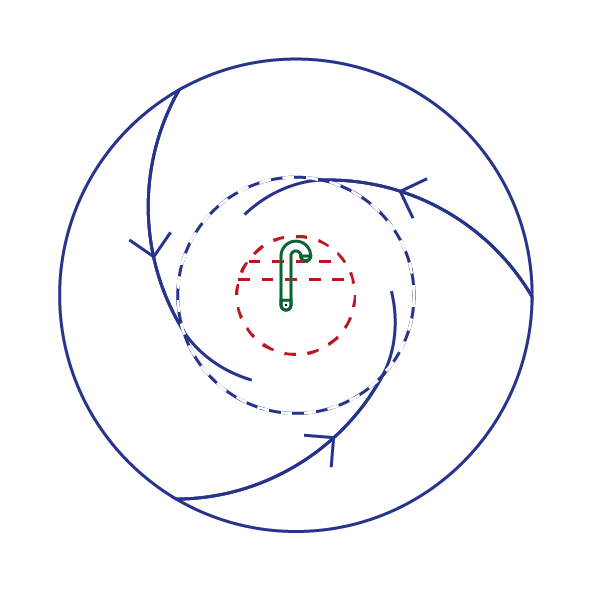}}
\subfigure{\includegraphics[width=40mm]{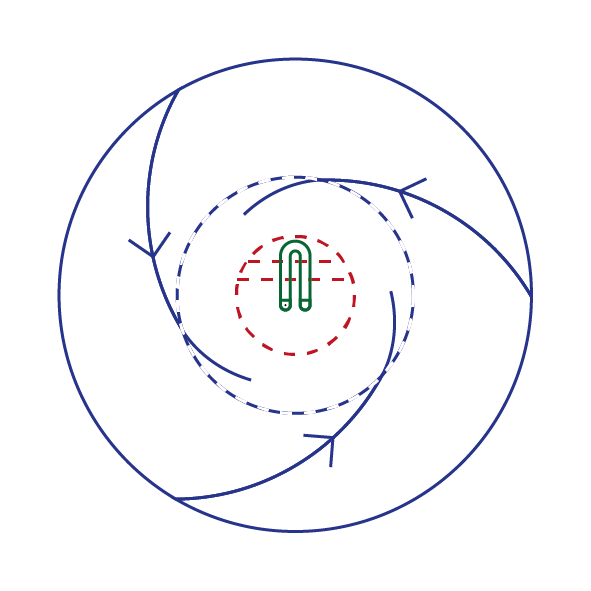}}
\caption{From the left: $h_0$; $h_t,\, 0<t<1/2$; $h_{1/2}$.} \label{fig:isotopy}
\end{figure}

\begin{figure}[htbp]
\centering
\subfigure{\includegraphics[width=40mm]{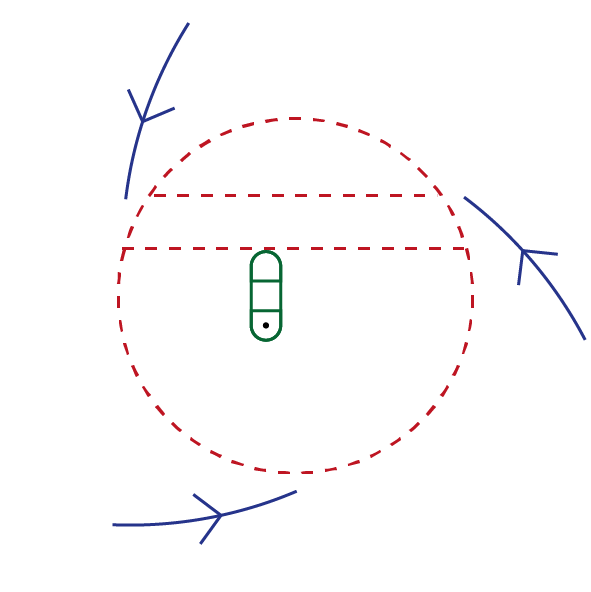}}
\subfigure{\includegraphics[width=40mm]{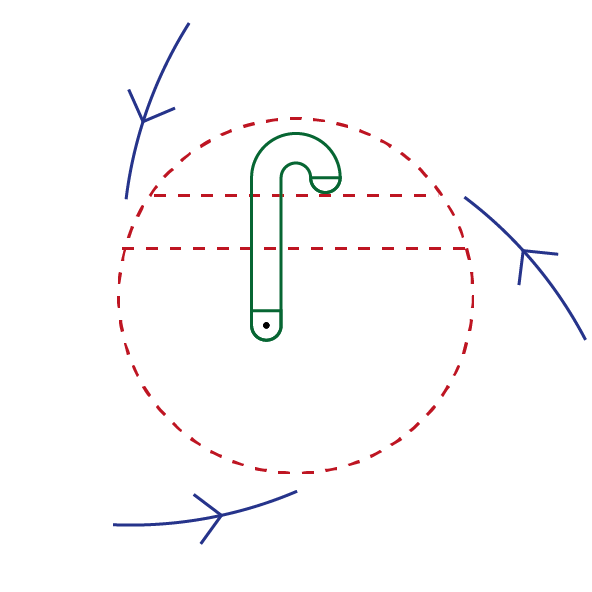}}
\subfigure{\includegraphics[width=40mm]{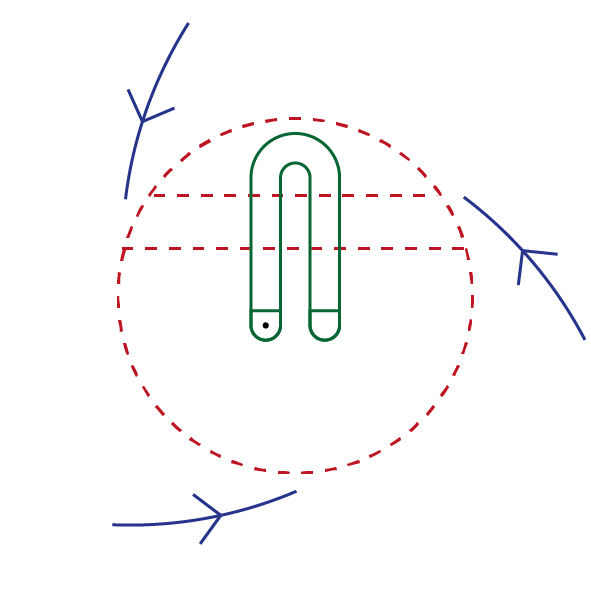}}
\caption{A zoom inside the disc $\mathbb{D}_{1/2}$. } \label{fig:isotopyzoom}
\end{figure}

Consider  $e_t:\mathbb{D}_{t}\to \mathbb{D}_1$ the homothety and define $\widehat{h}_t:\mathbb{D}_{t}\to \mathbb{D}_{t}$ by $\widehat{h}_t=e^{-1}_t\circ h_t\circ e_t$. Since $h_t$ and $e_t$ commute if $|z|\geq t/2$, then we have $\widehat{h}_t(z)=\lambda^{ws}_{f_0}z$ , hence $\widehat{h}_t$ can be extended to $\mathbb{R}^2$ with this formula.  Note also that $\sup_{z\in\mathbb{D}_{t} }\|D\widehat{h}_t\|=\sup_{z\in \mathbb{D}_1}\|Dh_t\|\leq 3$.  Next we fix $\delta>0$ and define $H:\mathbb{R}^4\to \mathbb{R}^4$ by the formula
\[
H(x_1,x_2,x_3,x_4)=\begin{cases}
  (\lambda^s_{f_0}x_1,\widehat{h}_{\frac{x^2_1+x^2_4}{\delta^2}}(x_2,x_3),\lambda^u_{f_0}x_4), &\text{ if } (x_1,x_2,x_3,x_4)\in (-\delta,\delta)^4\\
(\lambda^s_{f_0}x_1,\widehat{h}_1(x_2,x_3),\lambda^u_{f_0}x_4),   & \text{ if } (x_1,x_2,x_3,x_4)\not\in (-\delta,\delta)^4
\end{cases}.
\]

\begin{figure}
  \centering
    \includegraphics{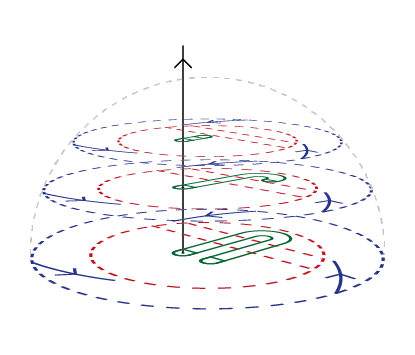}
  \caption{$H$ restricted to $x_1=0$.}
  \label{fig:completa}
\end{figure}

It follows that $H$ is differentiable.  Consider $p=0\in \mathbb{T}^4$ and take a chart $\phi: U\to (-2\delta,2\delta)^4$ around $p$ such that
$\phi \circ f_0\circ \phi^{-1}(\lambda^s_{f_0}x_1,\widehat{h}_1(x_2,x_3),\lambda^u_{f_0}x_4)$.  Finally, define  $f:\mathbb{T}^4
\to \mathbb{T}^4$ by
\[
f(q)=\begin{cases}
f_0(q), & \text{ if } q\not\in U\\
\phi^{-1}\circ H\circ\phi(q), & \text{ if } q\in U
\end{cases}.
\]
By construction, $f$ is isotopic to $f_0$, consequently isotopic to $A$. Let $\pi^{\sigma}: T\mathbb{T}^4\to T\mathbb{T}^4$ the projection into $E^{\sigma}_{f_0}$ for $\sigma\in\{s,u,c\}$ and define the family of cones
\begin{align*}
C^{s}:=\{v\in T\mathbb{T}^4:\|\pi^s(v)\|\geq\|\pi^u(v)+\pi^c(v)\|\},\\
C^{u}:=\{v\in T\mathbb{T}^4:\|\pi^u(v)\|\geq\|\pi^s(v)+\pi^c(v)\|\}.
\end{align*}By choosing $\delta>0$ small, one can guarantee that $D_qf(C^u(q))$ is properly contained in $C^u(fq)$, and similarly $D_qf^{-1}(C^s(fq))$ is properly contained in $C^s(q)$, for every $q\in\mathbb{T}^4$. It follows as in \cite{contributions} that $f$ is partially hyerpbolic, with center bundle $E^c_f=E^c_{f_0}$. This establishes \ref{Euno}. 

Next we tackle \ref{Edos}.

\begin{lemma} It holds $\htop(f)\geq\log|\lambda^u_A|+\log 2$.

\end{lemma}

\begin{proof}
Let us fix $\epsilon>0$ and $n\geq 1$. Consider a set $E(n,\epsilon)$ a maximal $(n,\epsilon)$-separated set inside the horseshoe: by hyperbolicity, it is no loss of generality to assume that $E(n,\epsilon)=f^{-n} E(\epsilon)$, where $E(\epsilon)$ is contained in an unstable manifold of the horseshoe inside $W^c_f(p)$. Consider $L>0$ such that for every $q,q'\in E(\epsilon)$ it holds
\[
\textrm{dist}_{\textrm{Haus}}(W^u_f(q;L), W^u_f(q';L))\geq \frac{\epsilon}{2}.
\]
Here $\textrm{dist}_{\textrm{Haus}}$ denotes the Hausdorff distance between the local leaves. For each $q\in E(\epsilon)$ we consider an $(n,\epsilon)$-separated set inside $W^u_f(q;L)$, which we denote by $\tilde{E}(q,n,\epsilon)$.
Note that 
\[
E=\bigcup_{q\in E(\epsilon)} f^{-n}\tilde{E}(q,n,\epsilon)
\]
is an $(2n,\frac{\epsilon}{2})$-separated set, and its cardinality satisfies 
\[
\#E\geq\#E(n,\epsilon)\cdot \min_{q\in E(\epsilon)} \#\tilde{E}(q,n,\epsilon).
\]
It remains to estimate the cardinality of $\tilde{E}(q,n,\epsilon)$: by the definition of $H$, and using that $E^u_f$ is one dimensional and the fact that $W^u_A$ are lines, we deduce that 
\[
\#\tilde{E}(q,n,\epsilon)\geq (\lambda^u_A)^n\left(\frac{L}{\epsilon}\right).
\]
In the end 
\[
\#E\geq \#E(n,\epsilon)(\lambda^u_A)^n\left(\frac{L}{\epsilon}\right),
\]
and then
\[
 \limsup_{n\to\infty}\frac{1}{n} \log \#E\geq \log |\lambda^u_A|+\limsup_{n\to\infty}\frac{1}{n} \log \#E(n,\epsilon)
\]
and since $\displaystyle\lim_{\epsilon\to 0}\limsup_{n\to\infty}\frac{1}{n} \log \#E(n,\epsilon)=\log 2,$ it follows our claim.
\end{proof}

It remains to show \ref{Etres}, that is, $f$ is robustly transitive: equivalently, we will show that $g=f^{-1}$ is robustly transitive. The argument is completely analogous to the one presented in the proof of Theorem B in \cite{contributions}, and is based in the following three facts.
\begin{enumerate}
	\item For every $x\in M$, if $I\subset W^u_g(x)$ is an interval, then there exists $y\in I$ and $n_0$ such that for $n\geq n_0$ it holds $g^n(y)\not\in U$ (the `keep-away lemma', see Lemma 5.2 in \cite{contributions} and compare Lemma A.4.2 in \cite{RodriguezHertz2008}).
	\item $W^c_g$ is minimal. This follows since $W^c_g=W^c_A$.
	\item As a consequence of the above, if $V\subset M$ is open, then there exists $R>0$  such that for every $x\in M$ it holds $W^c(x;R)\cap V\neq \emptyset$. 
\end{enumerate}

In particular $\mathcal{F}^c_g$ is plaque expansive (since it is differentiable, by Theorem 7.2 in \cite{HPS1977}). It follows that all the above properties are robust, namely if $g'$ is $\mathcal{C}^1$ close to $g$ then it satisfies the corresponding $(1),(2)$ and $(3)$ properties (for $(3)$ use Theorem 7.1 in \cite{HPS1977}).

Now take $V_1, V_2$ non-empty open sets and consider $I\subset V_1$ an unstable interval inside  $W^u_{g}$. Pick $y\in I$ satisfying $(1)$ and take $r>0$ such that $W^c(y;r)\subset V_1$. For a set $A$ included in a center leaf denote by 
\[
l(A):=\sup\{\delta>0:\exists z\in A: W^c_g(z;\delta)\subset A\}.
\]
Arguing as in \cite{contributions} page 394 (due to the fact that $Dg|E^c_g$ is expanding outside $U$) one shows that $l(g^n(W^c(y;r)))\xrightarrow[n\to\infty]{}\infty$,  hence for some $n_1$ we have $n\geq n_1$ implies $g^nW^c(y;r)\supset W^c(g^ny;R)$, where $R=R(V_2)$ is given by (3). In particular, $g^n(V_1)\cap V_2\neq \emptyset$. Hence, $g$ (and thus $f$) is robustly transitive\footnote{In fact, robustly topologically mixing.}.

\begin{remark}
It is simple to check that during the isotopy between $f$ and $A$ all maps are partially hyperbolic. It follows that $\uFcs,\uFu$ and $\uFcu, \uFcs$ have GPS \cite{Fisher2014}.
\end{remark}

\medskip

\section{Acknowledgments}

We are very indebted to the referees, in particular one of them pointed out the necessity of a more careful argument in the proof of Corollary \ref{estructuraclase}; while trying to address this problem we were able to remove several conditions appearing in the previous version, and present a cleaner Theorem \ref{teo:A}. The authors would also like to thank Rafael Potrie for his very useful feedback. 

\medskip

\section{Appendix: proof of Theorem \AE}

Since $f$ is a DA we consider the Franks' semi-conjugacy $h:\mathbb{T}^2\to\mathbb{T}^2$ such that $h\circ f=A\circ h$. As explained at the end of Section 3, to establish the equality $\htop(f) = \htop(A)$ it suffices to show that for every $x\in \mathbb{T}^2$, $\htop(f,h^{-1}(x))=0$.

We are assuming that each $f_t$ preserves an splitting of the form $T\mathbb{T}^2=E^c_t\oplus E^u_t$ where $E^u_t$ is uniformly expanded by the action of $Df_t$. Let $\Fu_t=\{W^u_1(x):x\in\mathbb{T}^2\}$ the invariant foliation tangent to $E^u_1$ given by the classical stable manifold theorem.

On the other hand, in Section 4 of  \cite{Potrie15}  it is proven that under the current hypotheses, the bundle $E^c_1$ is also integrable to an invariant foliation  $\Fc_1=\{W^c_1(x):x\in\mathbb{T}^2\}$. Lift everything to $\mathbb{R}^2$ and to simplify denote the lifted objects with the same letters. The classical Poincar\'e-Bendixon theorem implies that each $W^c_t(x)$ is homeomorphic to line: otherwise, there would be a singularity of $\Fu_1$, which is absurd. It also follows, again by \cite{Potrie15} that the lifted foliations are (quasi-isometric and) for every $x\in \mathbb{R}^2$ there exist an homeomorphism  between $W^u_1(x)\times W^c_1(x)$ sending the horizontal lines to $\Fu_1$ and the vertical lines to $\Fc_1$.

But then, proceeding exactly as in the proof of Theorem \ref{teo:A} we deduce that for every $x\in \mathbb{R}^2$ the set $h^{-1}(x)$ is an interval $I$ contained in a leaf $W^c_1(p)$, and such that $\{f^nI\}_n$ has uniformly bounded length. This implies that $\htop(f,I)=0$, finishing the proof.

\bibliographystyle{siam}
\bibliography{refcpec}
\end{document}